\documentclass[a4paper,11pt,reqno,twoside]{amsart}

\usepackage{amsmath}
\usepackage{amsfonts}
\usepackage{amssymb}
\usepackage{amsthm}
\usepackage{color}
\usepackage{ifpdf}
\usepackage{array}
\usepackage{url}
\usepackage{multirow,bigdelim}
\usepackage{ascmac}

\numberwithin{equation}{section}

\newtheorem{theorem}{Theorem}[section]
\newtheorem{lemma}{Lemma}[section]

\newcommand*{\C}{\mathbb{C}}
\newcommand*{\R}{\mathbb{R}}

\newcommand{\comment}[1]{}
\title[On infinitely divisible distributions related to the Riemann hypothesis]%
      {On infinitely divisible distributions related to \\ the Riemann hypothesis} 
\author[T. Nakamura]{Takashi Nakamura}
\author[M. Suzuki]{Masatoshi Suzuki}
\date{Version of \today}
\AtBeginDocument{%
\begin{abstract}
For the Riemann zeta-function, 
we introduce a function such that it is a characteristic function of 
an infinitely divisible distribution on the real line 
if and only if the Riemann Hypothesis is true. 
\end{abstract}
\maketitle
}
\begin{document}

%
\section{Introduction} 
%

The Riemann zeta-function $\zeta(s)$ is a function of a complex variable 
$s=\sigma+it$ ($\sigma, t \in \R$) defined by the Dirichlet series 
\begin{equation*}
\zeta(s) := \sum_{n=1}^{\infty} \frac{1}{n^s}
\end{equation*}
when $\Re(s)=\sigma>1$ and extends meromorphically to $\C$. 
Since $\zeta(s)$ is holomorphic except for a simple pole at $s=1$,  
the Riemann xi-function
\begin{equation} \label{eq_101}
\xi(s) := \frac{1}{2}s(s-1)\pi^{-s/2}\Gamma\left(\frac{s}{2}\right) \zeta(s)
\end{equation}
is an entire function, where $\Gamma(s)$ is the Gamma function 
(\cite[Theorem 2.1]{Tit86}). 
The Riemann Hypothesis (RH, for short), a famous open problem, 
claims that 
all non-real zeros of $\zeta(s)$ lie on the critical line $\Re(s)=1/2$,  
which is equivalent to all the zeros of $\xi(1/2-iz)$ being real.
The main result of the present paper states an equivalence condition for the RH 
in terms of infinitely divisible distributions in probability theory. 

A distribution (or probability measure) $\mu$ on $\R$ is infinitely divisible 
if there exists a distribution $\mu_n$ on $\R$ 
such that 
$\mu=\mu_n \ast \dots \ast \mu_n$ ($n$-fold) for every positive integer $n$. 
For any infinitely divisible distribution $\mu$, 
there exists a triplet $(a,b,\nu)$ 
consisting of $a \in \R_{\geq 0}$, $b \in \R$ 
and a measure $\nu$ on $\R$ 
such that the characteristic function 
\begin{equation*}
\widehat{\mu}(t)=\int_{-\infty}^{\infty} e^{itx} \mu(dx)
\end{equation*}
has the L{\'e}vy--Khintchine formula 
\begin{equation} \label{LK_1}
\widehat{\mu}(t) = \exp\left[
- \frac{1}{2} a t^2 + i b t 
+ \int_{-\infty}^{\infty}\left(
e^{it\lambda}  - 1 - \frac{it\lambda}{1+\lambda^2} 
\right) \nu(d \lambda)
\right], 
\end{equation}
\begin{equation} \label{LK_2}
\nu(\{0\})=0, \quad \int_{-\infty}^{\infty} \min(1,\lambda^2) \,\nu(d\lambda) < \infty
\end{equation}
(\cite[Theorem 8.1 and Remark 8.4]{Sa99}). 
The measure $\nu$ is referred to as the L{\'e}vy measure for $\mu$. 
If the L{\'e}vy measure $\nu$ satisfies $\int_{|\lambda|\leq 1} |\lambda|\,\nu(d\lambda)<\infty$, 
then \eqref{LK_1} can be rewritten as
\begin{equation} \label{LK_3}
\widehat{\mu}(t) = \exp\left[
- \frac{1}{2} a t^2 + i b_0 t 
+ \int_{-\infty}^{\infty}(e^{it\lambda}  - 1) \, \nu(d \lambda)
\right]
\end{equation}
(\cite[(8.7)]{Sa99}). 

The relation between the RH and an infinitely divisible distribution  
is formulated using the function $g_\zeta(t)$ defined by 
\begin{equation} \label{eq_105} 
\aligned 
g_\zeta(t)
& := - 4\,(e^{t/2}+e^{-t/2}-2) + \sum_{n \leq e^t} \frac{\Lambda(n)}{\sqrt{n}}(t-\log n) \\
& \quad -
 \frac{t}{2}\left( \psi(1/4) - \log \pi \right)
+
\frac{1}{4}\left( e^{-t/2}\Phi(e^{-2t},2,1/4) - \Phi(1,2,1/4)  \right)
\endaligned 
\end{equation}
for non-negative $t$ and $g_\zeta(t):=g_\zeta(-t)$ for negative $t$, 
where $\Lambda(n)=\log p$ if $n$ is a power of a prime $p$  
and otherwise $\Lambda(n)=0$, 
$\psi(z)=\Gamma'(z)/\Gamma(z)$ is the digamma function, and 
$\Phi(z,s,a) = \sum_{n=0}^{\infty} (n+a)^{-s}z^n$ 
is the Hurwitz--Lerch zeta-function. 
The function $g_\zeta(t)$ was originally introduced in \cite[(1.1), (1.3), and (1.8)]{Su22} 
to study equivalent conditions for the RH 
in relation to Weil's positivity or Li's criterion, etc. 

Then the main results are stated as follows:

\begin{theorem} \label{thm_1} 
Let $g_\zeta(t)$ be the function on the real line defined as above.
Then the following two are equivalent: 
\begin{enumerate}
\item the Riemann Hypothesis is true; 
\item $\exp(g_\zeta(t))$ is 
a characteristic function of 
an infinitely divisible distribution on $\R$. 
\end{enumerate}
\end{theorem}

Under the RH, the L{\'e}vy measure of the corresponding infinitely divisible distribution 
is also explicitly described as follows.

\begin{theorem} \label{thm_2}
Let $g_\zeta(t)$ be as in Theorem \ref{thm_1}. 
Assume that the RH is true and define 
\begin{equation} \label{Levy_1}
\nu_\zeta(d\lambda) := \sum_{\gamma} \frac{m_\gamma}{\gamma^2} \, \delta_{-\gamma}(d\lambda),
\end{equation}
where $\gamma$ runs through all zeros of $\xi(1/2-iz)$ without counting multiplicity, 
$m_\gamma$ is the multiplicity of $\gamma$, 
and $\delta_x$ is the delta measure at $x$. 
Then the L{\'e}vy--Khintchine formula \eqref{LK_3} 
of $\exp(g_\zeta(t))$ holds for the triplet $(a,b_0,\nu)=(0,0,\nu_\zeta)$, 
that is, $\exp(g_\zeta(t))$ is a characteristic function of a compound Poisson distribution.
\end{theorem}

The above two results are proved in Section \ref{section_2}. 
The core of the proof is the equality 
\begin{equation} \label{eq_107}
g_\zeta(t)= \sum_{\gamma} m_\gamma \frac{e^{-i\gamma t}-1}{\gamma^2}
\end{equation}
proved in \cite[Theorem 1.1 (2)]{Su22}, 
where the sum and $m_\gamma$ have the same meaning as in \eqref{Levy_1}. 
We give a proof of \eqref{eq_107} in the appendix (Section \ref{section_app}) 
for the convenience of readers. 
\medskip

It has long been known that the normalized function $Z_\sigma(t):=\zeta(\sigma+it)/\zeta(\sigma)$ is an infinitely divisible characteristic function for $\sigma>1$. More precisely, the distribution generated by $Z_\sigma(t)$ is compound Poisson on $\R$ 
and
\begin{equation*} 
Z_\sigma(t) = \exp \left[ \int_{-\infty}^\infty (e^{it\lambda} - 1) N_{\sigma} (d\lambda) \right], \quad
N_{\sigma} (d\lambda) := \sum_{p \in \mathbb{P}} \sum_{r=1}^\infty \frac{p^{-r\sigma}}{r} \delta_{-r\log p} (d\lambda),
\end{equation*}
where $\mathbb{P}$ denotes the set of all prime numbers (\cite[Example 17.6]{GK68}). 
Lin and Hu generalized this result for $\sum_{n=1}^\infty c(n) n^{-s}$, where $c(n)$ are completely multiplicative non-negative coefficients and $c(n) > 0$ for some $n \ge 2$ (\cite[Theorem 2]{LH01}).
Other examples of studies dealing with the Riemann zeta-function 
from a probabilistic point of view,
we cite papers reviewed in Biane--Pitman--Yor~\cite{BPY01}, Lagarias--Rains~\cite{LaRa03}, 
and Nakamura~\cite{Na15b}. 
In particular, Nakamura  discovered a direct relation between infinitely divisible distributions and the RH. 
He proved that the RH is true if and only if 
$\Xi_\sigma(t):=\xi(\sigma+it)/\xi(\sigma)$ 
is a characteristic function of 
a pretended-infinitely divisible distribution on $\R$ 
for every $1/2 < \sigma <1$, 
where a pretended-infinitely divisible distribution means a 
distribution satisfies \eqref{LK_1} 
but does not necessarily assume \eqref{LK_2}, 
and the measure $\nu$ is a signed measure on $\R$. 

The relation between the Riemann zeta-function and probability theory 
in the above literature is based on the Dirichlet series representation of $\zeta(s)$ 
or the integral representation of $\xi(s)$. 
In contrast, the results in the present paper are based on 
the integral representation of the logarithmic derivative of $\xi(s)$ in Lemma \ref{lem_2} below. 
This causes differences from known results.

\section{Proof of Theorems \ref{thm_1} and \ref{thm_2}} \label{section_2}

We start the proof with \eqref{eq_107}, and note the following. 

\begin{lemma} \label{lem_1}
The right-hand side of \eqref{eq_107} converges absolutely and uniformly 
on any compact subsets of $\R$. 
\end{lemma}
\begin{proof}
Let $K$ be a compact subset of $\R$. 
Then $|e^{-i\gamma t}-1|$ is uniformly bounded in $\gamma$ and $t \in K$, 
since all zeros of $\xi(1/2-iz)$ lie in the horizontal strip $|\Im(z)| \leq 1/2$ (\cite[\S2.12]{Tit86}). 
Further, $\sum_\gamma m_\gamma |\gamma|^{-\alpha}$ 
converges for any $\alpha>1$, 
since $\xi(1/2-iz)$ is an entire function of order one (\cite[Theorem 2.12]{Tit86}). 
Hence the sum on the right-hand side of \eqref{eq_107} 
converges absolutely and uniformly for $t \in K$. 
\end{proof}

From the above lemma, the right-hand side of \eqref{eq_107} 
defines a continuous function on the real line. 
By the functional equation 
\begin{equation} \label{eq_201}
\xi(s)=\xi(1-s)
\end{equation} 
(\cite[(2.1.13)]{Tit86}), 
if $\gamma$ is a zero of $\xi(1/2-iz)$, 
then $-\gamma$ is also a zero with the same multiplicity. 
Therefore, the right-hand side of \eqref{eq_107} is an even real-valued function. 
This is consistent with the definition of $g_\zeta(t)$. 

\begin{proof}[Proof of (1)$\Rightarrow$(2) in Theorem \ref{thm_1} and  Theorem \ref{thm_2}]
Since (1) is assumed, all zeros $\gamma$ of $\xi(1/2-iz)$ are real. 
On the other hand, $\xi(1/2) \not=0$, 
because 
$(1-2^{1-\sigma})\zeta(\sigma)>0$  
and $\Gamma(\sigma/2)>0$ for any positive real numbers $\sigma$ 
(\cite[\S2.12]{Tit86}). 
Therefore, \eqref{Levy_1} defines a measure on $\R$ such that  
the formula 
\begin{equation*} 
\exp(g_\zeta(t)) 
= \exp\left[ \int_{-\infty}^{\infty}(
e^{it\lambda}  - 1) \nu_\zeta(d \lambda) \right]
\end{equation*} 
holds on $\R$, $\nu_\zeta(\{0\})=0$ and 
$\int_{|\lambda| \leq 1}|\lambda|\nu_\zeta(d\lambda)<\infty$ by \eqref{eq_107}. 
Moreover, 
\begin{equation*} 
\int_{-\infty}^{\infty}\min(1,\lambda^2)\,\nu_\zeta(d\lambda)<\infty
\end{equation*} 
by the convergence of $\sum_\gamma m_\gamma |\gamma|^{-2}$ 
in the proof of Lemma \ref{lem_1}. 
Hence, there exists an infinitely divisible distribution $\mu$ 
whose characteristic function is given by $\exp(g_\zeta(t))$ 
with the characteristic triplet $(0,0,\nu_\zeta)$
by \cite[Theorem 8.1 (iii)]{Sa99}. 
\end{proof}

To prove (2)$\Rightarrow$(1), we need the following integral formula.

\begin{lemma} \label{lem_2}
The Fourier integral formula
\begin{equation} \label{eq_202}
\int_{0}^{\infty} g_\zeta(t) \,e^{izt} \, dt = \frac{1}{z^2}\frac{\xi'}{\xi}\left(\frac{1}{2}-iz\right)
\end{equation} 
holds if $\Im(z)>1/2$. 
\end{lemma}
\begin{proof} We start with \eqref{eq_107} again.  
Since $\xi(s)$ is order one entire function, Hadamard's factorization theorem gives
\[
\xi\left(\frac{1}{2}-iz\right) = e^{a+bz} \prod_\gamma \left[\left(1-\frac{z}{\gamma} \right)
e^{\frac{z}{\gamma}} \right]^{m_\gamma}. 
\]
Taking the logarithmic derivative of both sides, 
\begin{equation} \label{eq_203}
\frac{\xi'}{\xi}\left(\frac{1}{2}-iz\right)
= ib + i\sum_\gamma m_\gamma \left( \frac{1}{z-\gamma} + \frac{1}{\gamma} \right), 
\end{equation}
where the sum on the right-hand side converges absolutely and uniformly on every
compact subset of $\C$ outside the zeros $\gamma$. 
Substituting $z=0$ in \eqref{eq_203}, we have $ib=(\xi'/\xi)(1/2)$. 
On the other hand, taking the logarithmic derivative of \eqref{eq_201}, 
\begin{equation} \label{eq_204}
\frac{\xi'}{\xi}(s) = - \frac{\xi'}{\xi}(1-s). 
\end{equation}
Substituting $s=1/2$ in \eqref{eq_204}, we get $(\xi'/\xi)(1/2)=0$, 
thus $ib=(\xi'/\xi)(1/2)=0$. For each term on the right-hand side of \eqref{eq_203}, 
we have 
\begin{equation*} 
\frac{i}{z^2}  \left( \frac{1}{z-\gamma} + \frac{1}{\gamma} \right)
=
\int_{0}^{\infty} \frac{e^{-it\gamma}-1}{\gamma^2}\, e^{izt} \, dt, 
\quad \Im(z)>\Im(\gamma)
\end{equation*}
by direct calculation of the right-hand side, 
where $|\Im(\gamma)| \leq 1/2$ (\cite[Theorem 2.12]{Tit86}). 
Therefore, \eqref{eq_202} is obtained by interchanging summation and integration, 
which is justified by Lemma \ref{lem_1}.
\end{proof}

\begin{proof}[Proof of (2)$\Rightarrow$(1) in Theorem \ref{thm_1}] 
We have $\exp(g_\zeta(t))=\widehat{\mu}(t)$ 
for some infinitely divisible distribution $\mu$ on the real line by (2). 
Then $\exp(g_\zeta(t))=\widehat{\mu}(t)=\widehat{\mu}(-t)$, 
since $g_\zeta(t)$ is an even function independent of the RH 
as mentioned before Theorem \ref{thm_1}.  
The second equality implies 
$\widehat{\mu}(t)=\int_{-\infty}^{\infty}\cos(tx)\mu(dx)$. 
Therefore, $g_\zeta(t)\,(=\log \widehat{\mu}(t))$ is real-valued. 
Moreover, $g_\zeta(t)$ is non-positive on $\R$, 
since $|\widehat{\mu}(t)| \leq 1$ holds because $\mu$ is a distribution. 

Since $g_\zeta(t)$ is non-positive, 
the pure imaginary point $iy$ on the horizontal axis $\Im(z)=y$ of convergence 
of $f(z)=\int_{0}^{\infty} g_\zeta(t) \, e^{izt} \, dt$ 
is a singularity of $f(z)$ by 
\cite[Theorem 5b in Chap. II]{Wi41}. 
On the other hand, $(\xi'/\xi)(1/2-iz)$ has no singularity 
on the positive imaginary axis $i\R_{>0}$, 
since $\xi(s)$ is holomorphic and has no zeros on the positive real line $\R_{>0}$ 
(\cite[\S2.12]{Tit86}). 
Hence $f(z)$ converges 
in the upper half-plane $\{z\,|\,\Im(z)>0\}$ 
and defines a holomorphic function there. 
Therefore,  $\xi(1/2-iz)$ has no zeros in the upper half-plane  
by \eqref{eq_202}, 
and then it also has no zeros in the lower half-plane $\{z\,|\,\Im(z)<0\}$ 
by \eqref{eq_201}.  
Thus all zeros of $\xi(1/2-iz)$ are real, it is the RH. 
\end{proof}

\section{Appendix: Proof of \eqref{eq_107}} \label{section_app} 

We have shown \eqref{eq_202} using the right-hand side of \eqref{eq_107}. 
Therefore, it is sufficient to prove that the right-hand side of \eqref{eq_105}  
satisfies \eqref{eq_202} by the uniqueness of the Fourier transform. 
Taking the logarithmic derivative of \eqref{eq_101} and 
using the series expansion 
\begin{equation*} 
\frac{\zeta'(s)}{\zeta(s)} = -\sum_{n=2}^{\infty} \frac{\Lambda(n)}{n^s} \quad (\Re(s)>1)
\end{equation*}
(\cite[(1.1.8)]{Tit86}), we have
\begin{equation*} 
\aligned 
\frac{\xi'}{\xi}(s) 
& = \frac{1}{s-1} +\frac{1}{s} - \frac{1}{2}\log \pi 
+ \frac{1}{2}\psi(s/2) 
- \sum_{n=2}^{\infty} \frac{\Lambda(n)}{n^s}
\endaligned 
\end{equation*}
for $\Re(s)>1$. Writing $1/2-iz$ as $s$, we get
\begin{equation*} 
-\int_{0}^{\infty} 4(e^{t/2}+e^{-t/2}-2)\, e^{izt} \, dt 
= \frac{1}{z^2}\,\left(\frac{1}{s-1}+\frac{1}{s}\right) \quad \text{if}~\Im(z)>1/2, 
\end{equation*}
\begin{equation*} 
\int_{0}^{\infty} t\,e^{izt} \, dt 
= -\frac{1}{z^2} \quad \text{if}~\Im(z)>0, \quad \text{and}
\end{equation*}
\begin{equation*} 
\int_{\log n}^{\infty} \frac{(t-\log n)}{\sqrt{n}} \, e^{izt} \, dt 
= -\frac{1}{z^2} \, n^{-s} \quad \text{if}~\Im(z)>0
\end{equation*}
by direct and simple calculation. The third equality leads to 
\begin{equation*} 
\aligned 
\int_{0}^{\infty}
\left(\sum_{n \leq e^t} \frac{\Lambda(n)}{\sqrt{n}}(t-\log n)\right) e^{izt} \, dt 
& = 
\frac{1}{z^2}\left(-\sum_{n=2}^{\infty} \frac{\Lambda(n)}{n^s} \right)\quad  \text{if}~
\Im(z)>1/2, 
\endaligned 
\end{equation*} 
where the condition $\Im(z)>1/2$ is due to the convergence of the series on the right-hand side.
Hence, the proof is completed if it is proved that 
\begin{equation} \label{eq_301}
\int_{0}^{\infty} 
\frac{1}{4}
\left( e^{-t/2}\Phi(e^{-2t},2,1/4) - \Phi(1,2,1/4) \right) e^{izt} \, dt 
=
 \frac{1}{2z^2} \left(
\psi(s/2)-\psi(1/4)
\right)
\end{equation}
holds for $\Im(z)>0$. To prove this, we recall the series expansion 
\begin{equation} \label{eq_302}
\psi(w) = -\gamma_0 - \sum_{n=0}^{\infty}
\left( \frac{1}{w+n} - \frac{1}{n+1} \right), 
\end{equation}
where $\gamma_0$ is the Euler--Mascheroni constant (\cite[(1.3.16)]{Le72}). 
Besides, we have 
\begin{equation*} 
\int_{0}^{\infty} \frac{e^{-2(n+\frac{1}{4}) t}-1}{2(n+\frac{1}{4})^2}\, e^{izt} \, dt 
= -\frac{1}{z^2} \left( \frac{1}{s/2+n} - \frac{1}{n+1/4} \right) 
\end{equation*}
for non-negative integers $n$ and $\Im(z)>0$ 
by direct and simple calculation.  
Therefore, equality \eqref{eq_302} implies
\begin{equation*} 
\aligned 
\int_{0}^{\infty} & 
\frac{1}{4}\left(
e^{-t/2} 
\Phi(e^{-2t},2,1/4)
-\Phi(1,2,1/4)
\right) e^{izt}\, dt \\
& = \frac{1}{2}
\int_{0}^{\infty} 
\left( 
\sum_{n=0}^{\infty}
\frac{e^{-2(n+\frac{1}{4}) t}-1}{2(n+1/4)^2}
\right) e^{izt}\, dt  
 = -\frac{1}{2z^2}  \sum_{n=0}^{\infty} \left( \frac{1}{s/2+n} - \frac{1}{n+1/4} \right) \\
& = - \frac{1}{2z^2}  \sum_{n=0}^{\infty} \left( \frac{1}{s/2+n} - \frac{1}{n+1} \right)
-
\frac{1}{2z^2}  \sum_{n=0}^{\infty} \left(  \frac{1}{n+1} - \frac{1}{n+1/4}   \right) \\
& = 
 \frac{1}{2z^2} \bigl(\psi(s/2)+ \gamma_0\bigr) 
- \frac{1}{2z^2} \bigl(\psi(1/4)+ \gamma_0\bigr) ,
\endaligned 
\end{equation*}
where $s=1/2-iz$. 
Hence, we obtain \eqref{eq_301}. 
\medskip

\noindent
{\bf Acknowledgments}~
The first and second authors were supported by JSPS KAKENHI 
Grant Number JP22K03276 and JP17K05163--JP23K03050, respectively. 
This work was also supported by the Research Institute for Mathematical Sciences, 
an International Joint Usage/Research Center located in Kyoto University.

%

%
\end{document}